\documentclass{amsart}
\usepackage{amssymb}
\usepackage{amsmath}
\usepackage{amsfonts}

\newtheorem{theorem}{Theorem}
\theoremstyle{plain}

\newtheorem{definition}{Definition}
\newtheorem{example}{Example}

\numberwithin{equation}{section}
\input{tcilatex}

\begin{document}
\title[Continuous fiber bundles of $C^{\ast }$-algebras]{On continuous fiber
bundles of $C^{\ast }$-algebras over Stonean compact}
\dedicatory{Dedicated to the memory of Professor George Bachman (Polytechnic
University, NY, USA)}
\author{Alexander A. Katz}
\address{Alexander A. Katz, Department of Mathematics and Computer Science,
St. John's College of Liberal Arts and Sciences, St. John's University, 300
Howard Avenue, DaSilva Academic Center 314, Staten Island, NY 10301, USA}
\email{katza@stjohns.edu}
\author{Roman Kushnir}
\address{Roman Kushnir, Department of Mathematical Sciences, University of
South Africa, P.O. Box 392, Pretoria 0003, Republic of South Africa}
\curraddr{Roman Kushnir, Department of Mathematics and Computer Studies,
York College, CUNY, 94-20 Guy R. Brewer Boulevard, Jamaica, NY 11451, USA}
\email{kushnir\_roman@yahoo.com}
\date{June 26-27, 2010}
\subjclass[2010]{Primary 46L05; Secondary 46A19.}
\keywords{Banach-Kantorovich space, $C^{\ast }$-algebra over the algebra $%
C_{\infty }(Q,%
\mathbb{C}
),$ Stonean compact, Continuous Fiber Bundle of $C^{\ast }$-algebras.}

\begin{abstract}
We introduce $C^{\ast }$-algebras over $C_{\infty }(Q,%
\mathbb{C}
)$ as Banach-Kantorovich $^{\ast }$-algebras over the algebra $C_{\infty }(Q,%
\mathbb{C}
)$ of extended continuous complex-valued functions, defined on comeager
subsets of Stonean compact $Q$, whose norm satisfies conditions similar to
the axioms of $C^{\ast }$-algebras, and show that such algebras can be
uniquely up to a $Q$-$C^{\ast }$-isomorphism represented by means of a
continuous complete fiber bundle of $C^{\ast }$-algebras over $Q$.
\end{abstract}

\maketitle

\section{Introduction}

In the present paper we introduce a sort of a continuous v.s. measurable
theory of Ganiev and Chilin (see \cite{GanievChilin2003}) by studying
Banach-Kantorovich spaces which are simultaneously $^{\ast }$-algebras with
norm satisfying $C^{\ast }$-conditions. Algebras of this sort are modules
over the algebra $C_{\infty }(Q,%
\mathbb{C}
)$ of extended complex-valued continuous functions defined on comeager
subsets of a Stonean (extremally disconnected) compact $Q.$ Therefore, it is
natural to call them $C^{\ast }$\textbf{-algebras over }$C_{\infty }(Q,%
\mathbb{C}
)$, where $Q$ is a Stonean compact. The $C^{\ast }$-algebras over $C_{\infty
}(Q,%
\mathbb{C}
)$, where $Q$ is a Stonean compact, exhibit interesting examples of
Banach-Kantorovich spaces, the theory of whose is well developed by now (see
for example the monograph \cite{Kusraev2000} by Kusraev)

The aim of the paper is, following the general idea of representation of
Banach-Kantorovich spaces as continuous fiber bundles of Banach spaces (see 
\cite{KusraevStrizhevsky1984}, \cite{Gutman1991}, \cite{Gutman1993}), to
give a representation of $C^{\ast }$-algebras over $C_{\infty }(Q,%
\mathbb{C}
)$, where $Q$ is a Stonean compact, as a continuous fiber bundle of $C^{\ast
}$-algebras over $Q,$ which allows to study them using the general theory of
continuous Banach fiber bundles.

The continuous Banach fiber bundle is a formalization of an intuitive idea
of a family of Banach spaces $\mathfrak{X}(q),$ $q\in Q,$ continuously
changing over $Q$. The invention of Banach bundles is usually credited to
von Neumann who in 1937 spelled out an idea of "variable" Banach spaces. The
formal descriptions were made about 1950 and associated with the names of
Godement, Kaplansky, Gelfand and Naimark, to name a few. Today the theory of
continuous Banach fiber bundles is a wide area of research. The reader can
become familiar with this theory using the monograph \cite{Gierz1982}\ by
Gierz. The idea to represent an analytical object as a space of sections of
a fiber bundle is not new in analysis. One can familiarize himself with it
using the review \cite{HofmannKeimel1977} by Hofmann and Keimel.

We will use without much further references the notions and results of:

\begin{itemize}
\item the theory of Kantorovich spaces and lattice-normed spaces from \cite%
{Vulikh1967}, \cite{Kusraev1985} and \cite{Kusraev2000};

\item the theory of $C^{\ast }$-algebras from \cite{Dixmier1977}, \cite%
{Murphy1990} and \cite{Pedersen1979};
\end{itemize}

\section{Preliminaries}

Let $Q$ be a compact Hausdorff topological space. The space $Q$ is called 
\textbf{Stonean} or \textbf{extremally} (\textbf{Quasi-stonean} or \textbf{%
quasi-extremally}) disconnected or simply \textit{extremal} (\textit{%
quasiextremal}) if the closure of an arbitrary open set (open $F_{\sigma }$%
-set) in it is open or, which is equivalent, the interior of an arbitrary
closed set (closed $G_{\delta }$-set) is closed.

A \textbf{Riesz space} or a \textbf{vector lattice} is an ordered vector
space which is as well a lattice. Thus, a vector lattice contains the least
upper bound (or supremum) and the greatest lower bound (or infimum) of each
finite subset. A \textbf{Kantorovich space} or, briefly, a $\mathbf{K}$%
\textbf{-space} is a vector lattice over the field $%
\mathbb{R}
,$ such that every order-bounded set in it has least upper and greatest
lower bounds. Sometimes a more precise term, \textit{(conditionally) order
complete vector lattice}, is employed instead of $K$-space. If, in a vector
lattice, least upper and greatest lower bounds exist only for countable
bounded sets, then it is called a $\mathbf{K}_{\sigma }$\textbf{-space}.
Each $K_{\sigma }$-space and, hence, a $K$-space are Archimedean. We say
that a $K$-space ($K_{\sigma }$-space) is \textbf{universally complete} or 
\textbf{extended} if every its subset (countable subset) composed of
pairwise disjoint elements is bounded.

An \textbf{ordered algebra} over a field $\mathbb{F}$ is an ordered vector
space $E$ over $\mathbb{F}$ which is simultaneously an algebra over the same
field and satisfies the following condition:

$(\ast )$ if $x,y\in E,$ and $\theta \in E$ is the zero element of $E$, and%
\begin{equation*}
x\geq \theta ,
\end{equation*}%
and 
\begin{equation*}
y\geq \theta ,
\end{equation*}%
then%
\begin{equation*}
xy\geq \theta .
\end{equation*}

To characterize the positive cone $E_{+}$ of an ordered algebra $E$, we must
add to what was said in $(\ast )$ the property 
\begin{equation*}
E_{+}\cdot E_{+}\subset E_{+}.
\end{equation*}

We say that $E$ is a \textbf{lattice-ordered algebra} if $E$ is a Riesz
space and an ordered algebra simultaneously. A lattice-ordered algebra is an 
$f$\textbf{-algebra} if, for all $a,x,y\in E_{+}$, the condition 
\begin{equation*}
x\wedge y=\theta ,
\end{equation*}%
implies that 
\begin{equation*}
(ax)\wedge y=\theta ,
\end{equation*}%
and%
\begin{equation*}
(xa)\wedge y=\theta .
\end{equation*}

An $f$-algebra $E$ is called \textbf{faithful} or \textbf{exact} if, for
arbitrary elements $x,y\in E,$%
\begin{equation*}
xy=\theta ,
\end{equation*}%
implies 
\begin{equation*}
\left\vert x\right\vert \wedge \left\vert y\right\vert =\theta .
\end{equation*}%
It is easy to show (see for example \cite{Kusraev2000}) that an $f$-algebra
is faithful if and only if it lacks nonzero nilpotent elements. The
faithfulness of an $f$-algebra is equivalent to absence of strictly positive
element with nonzero square.

A complex vector lattice is defined to be the complexification 
\begin{equation*}
E\dotplus iE,
\end{equation*}%
(with $i\in 
\mathbb{C}
$ standing for the imaginary unity) of a real vector lattice $E$. Often it
is additionally required that the modulus 
\begin{equation*}
\left\vert z\right\vert =sup\{\func{Re}(e^{i\eta }z):0\leq \eta \leq \pi \}
\end{equation*}%
exists for every element%
\begin{equation*}
z\in E\dotplus iE.
\end{equation*}

In the case of a $K$-space or an arbitrary Banach lattice this requirement
is automatically satisfied, since a complex $K$-space is the
complexification of a real $K$-space. When we talk about order properties of
a complex vector lattice 
\begin{equation*}
E\dotplus iE,
\end{equation*}%
we mean its real part $E$. The concepts of sublattice, ideal, band,
projection, etc. are naturally translated to the case of a complex vector
lattice by appropriate complexification.

The order of a vector lattice generates different kinds of convergence. Let $%
(\Delta ,\leq )$ be an upward-directed set. A net 
\begin{equation*}
(x_{\alpha })=(x_{\alpha })_{\alpha \in \Delta }\in E,
\end{equation*}%
is called increasing (decreasing) if 
\begin{equation*}
x_{\alpha }\leq x_{\beta },
\end{equation*}
\begin{equation*}
(x_{\beta }\leq x_{\alpha }),
\end{equation*}%
for 
\begin{equation*}
\alpha \leq \beta ,
\end{equation*}%
$\alpha ,\beta \in \Delta $.

It is said that a net $(x_{\alpha })$ $o$\textit{-converges} to an element $%
x\in E$ if there exists a decreasing net $(e_{\alpha })_{\alpha \in \Delta }$
in $E$ such that 
\begin{equation*}
\inf \{e_{\alpha }:\alpha \in \Delta \}=\theta ,
\end{equation*}%
and 
\begin{equation*}
\left\vert x_{\alpha }-x\right\vert \leq e_{\alpha },
\end{equation*}%
for all $\alpha \in \Delta $. In this case we call $x$ the $o$\textit{-limit}
of the net $(x_{\alpha })$ and write 
\begin{equation*}
x=o-\lim x_{\alpha },
\end{equation*}%
or 
\begin{equation*}
x_{\alpha }\overset{(o)}{\rightarrow }x.
\end{equation*}

In a $K$-space, we also introduce the \textbf{upper} and \textbf{lower} $o$%
\textbf{-limits} of an order bounded net by the formulae 
\begin{equation*}
\underset{\alpha \in \Delta }{\lim \sup }x_{\alpha }=\overline{\underset{%
\alpha \in \Delta }{\lim }}x_{\alpha }=\underset{\alpha \in \Delta }{\inf }%
\underset{\beta \geq \alpha }{\sup }x_{\beta },
\end{equation*}

and 
\begin{equation*}
\underset{\alpha \in \Delta }{\lim \inf }x_{\alpha }=\underset{\alpha \in
\Delta }{\underline{\lim }}x_{\alpha }=\underset{\alpha \in \Delta }{\sup }%
\underset{\beta \geq \alpha }{\inf }x_{\beta }.
\end{equation*}

One can see that 
\begin{equation*}
x=o-\lim x_{\alpha },
\end{equation*}%
is equivalent to 
\begin{equation*}
\underset{\alpha \in \Delta }{\lim \sup }x_{\alpha }=x=\underset{\alpha \in
\Delta }{\lim \inf }x_{\alpha }.
\end{equation*}

We say that a net $(x_{\alpha })_{\alpha \in \Delta }$ \textbf{converges
relatively uniformly} or \textbf{converges with regulator} to $x\in E$ if
there exist an element $u$, such that 
\begin{equation*}
\theta \leq u\in E,
\end{equation*}%
called the \textbf{regulator} of convergence, and a numeric net $(\lambda
\alpha )_{\alpha \in \Delta }$ with the properties 
\begin{equation*}
\underset{\alpha \in \Delta }{\lim }\lambda _{\alpha }=0,
\end{equation*}%
and 
\begin{equation*}
\left\vert x_{\alpha }-x\right\vert \leq \lambda _{\alpha }u,
\end{equation*}%
$\alpha \in \Delta .$ The element $x$ is called the $r$\textbf{-limit} of
the net $(x_{\alpha }),$ and we denote it as 
\begin{equation*}
x=r-\underset{\alpha \in \Delta }{\lim }x_{\alpha },
\end{equation*}%
or 
\begin{equation*}
x_{\alpha }\overset{(r)}{\rightarrow }x.
\end{equation*}

The presence of $o$-convergence in a $K$-space justifies the definition of
the sum for an infinite family $(x_{\xi })_{\xi \in \Xi }$. Indeed, let 
\begin{equation*}
A=\mathcal{P}_{fin}(\Xi ),
\end{equation*}%
be the set of all finite subsets of $\Xi $. Given 
\begin{equation*}
\alpha =\{\xi _{1},...,\xi _{n}\}\in A,
\end{equation*}%
we denote%
\begin{equation*}
y_{\alpha }=x_{\xi _{1}}+...+x_{\xi _{n}}.
\end{equation*}

By doing so we obtain the net $(y_{\alpha })_{\alpha \in \Delta }$ which is
naturally ordered by inclusion. If there exists 
\begin{equation*}
x=o-\underset{\alpha \in \Delta }{\lim }y_{\alpha },
\end{equation*}%
then we call the element $x$ the $o$\textbf{-sum} of the family $(x_{\xi })$
and denote it by 
\begin{equation*}
x=o-\dsum\limits_{\xi \in \Xi }x_{\xi }=\dsum\limits_{\xi \in \Xi }x_{\xi }.
\end{equation*}

One can see that, for 
\begin{equation*}
x_{\xi }\geq \theta ,
\end{equation*}%
$\xi \in \Xi ,$ the $o$-sum of the family $(x_{\xi })$ exists if and only if
the net $(y_{\alpha })_{\alpha \in A}$ is order bounded; in this case 
\begin{equation*}
o-\dsum\limits_{\xi \in \Xi }x_{\xi }=\underset{\alpha \in A}{\sup }%
y_{\alpha }.
\end{equation*}

If the elements of the family $(x_{\xi })$ are pairwise disjoint then 
\begin{equation*}
o-\dsum\limits_{\xi \in \Xi }x_{\xi }=\underset{\xi \in \Xi }{\sup }x_{\xi
}^{+}-\underset{\xi \in \Xi }{\sup }x_{\xi }^{-}.
\end{equation*}

Every $K$-space is $o$-complete in the following sense:

If a net $(x_{\alpha })_{\alpha \in A}$ satisfies the condition 
\begin{equation*}
\lim \sup \left\vert x_{\alpha }-x_{\beta }\right\vert =\underset{\lambda
\in A}{\inf }\underset{\alpha ,\beta \geq \lambda }{\sup }\left\vert
x_{\alpha }-x_{\beta }\right\vert ,
\end{equation*}

then there is an element $x\in E$ such that 
\begin{equation*}
x=o-\lim x_{\alpha }.
\end{equation*}

The space of continuous functions taking infinite values on a nowhere dense
subsets of a topological space $Q$ plays an important role in the theory of
Riesz spaces. To introduce this space, we need some auxiliary facts.

Given a function 
\begin{equation*}
f:Q\rightarrow \overline{%
\mathbb{R}
},
\end{equation*}%
and a number $\lambda \in 
\mathbb{R}
,$ we denote 
\begin{equation*}
\{f<\lambda \}=\{q\in Q:f(q)<\lambda \},
\end{equation*}%
and 
\begin{equation*}
\{f\leq \lambda \}=\{q\in Q:f(q)\leq \lambda \}.
\end{equation*}

Let $Q$ be an arbitrary topological space, let $\Lambda $ be a dense set in $%
\mathbb{R}
,$ and let 
\begin{equation*}
\lambda \mapsto G_{\lambda },
\end{equation*}%
$\lambda \in \Lambda ,$ be an increasing mapping from $\Lambda $ into the
power-set $\mathcal{P}(Q)$ of the space $Q$, ordered by inclusion. We denote 
\begin{equation*}
\overline{%
\mathbb{R}
}=%
\mathbb{R}
\cup \{\pm \infty \}.
\end{equation*}%
Then the following assertions are equivalent (see for example \cite%
{Kusraev2000}):

$1).$ there exists a unique continuous function 
\begin{equation*}
f:Q\rightarrow \overline{%
\mathbb{R}
},
\end{equation*}%
such that 
\begin{equation*}
\{f<\lambda \}\subset G_{\lambda }\subset \{f\leq \lambda \},
\end{equation*}%
$\lambda \in \Lambda ;$

2). for arbitrary $\lambda ,\mu \in \Lambda $, the inequality 
\begin{equation*}
\lambda <\mu ,
\end{equation*}%
implies 
\begin{equation*}
cl(G_{\lambda })\subset int(G_{\mu }).
\end{equation*}

Let $Q$ be a Quasi-stonean space, let $Q_{0}$ be an open dense $F_{\sigma }$%
-subset in $Q$, and let 
\begin{equation*}
f_{0}:Q_{0}\rightarrow \overline{%
\mathbb{R}
},
\end{equation*}%
be a continuous function. There exists (see for example \cite{Kusraev2000})
a unique continuous function 
\begin{equation*}
f:Q\rightarrow \overline{%
\mathbb{R}
},
\end{equation*}%
such that 
\begin{equation*}
f(q)=f_{0}(q),
\end{equation*}%
for all $q\in Q_{0}$.

We denote by $C_{\infty }(Q,%
\mathbb{R}
)$ the set of all continuous functions 
\begin{equation*}
f:Q\rightarrow \overline{%
\mathbb{R}
},
\end{equation*}%
that may take values $\pm \infty $ only on a nowhere dense set. Introduce
some order on $C_{\infty }(Q,%
\mathbb{R}
)$ by putting 
\begin{equation*}
f\leq g,
\end{equation*}%
if and only if 
\begin{equation*}
f(q)\leq g(q),
\end{equation*}%
for all $q\in Q$. Further, take $f,g\in C_{\infty }(Q,%
\mathbb{R}
)$ and put 
\begin{equation*}
Q_{0}=\{\left\vert f\right\vert <\infty \}\cup \{\left\vert g\right\vert
<\infty \}.
\end{equation*}%
Then $Q_{0}$ is an open and dense $F_{\sigma }$-subset in $Q$. According to
aforementioned, there exists a unique function 
\begin{equation*}
h:Q\rightarrow \overline{%
\mathbb{R}
}
\end{equation*}%
such that 
\begin{equation*}
h(q)=f(q)+g(q),
\end{equation*}%
for $q\in Q_{0}$. This function $h$ is considered to be the sum of the
elements $f$ and $g$. The product of two arbitrary elements is defined in a
similar way. Identifying a number $\lambda $ with the function identically
equal to $\lambda $ on $Q$, we obtain the product of an arbitrary $f\in
C_{\infty }(Q,%
\mathbb{R}
)$ and $\lambda $ $\in $ $%
\mathbb{R}
$. It is easy to see that $C_{\infty }(Q,%
\mathbb{R}
)$ with the so-introduced operations and order is a vector lattice and
simultaneously a faithful $f$-algebra. It is well known that $C_{\infty }(Q,%
\mathbb{R}
)$ is a universally complete $K_{\sigma }$-space. The function identically
equal to unity is a ring and lattice unity. The base of the vector lattice $%
C_{\infty }(Q,%
\mathbb{R}
)$ is isomorphic to the Boolean algebra of all regular open (closed) subsets
of the compact space $Q$. If the compact space $Q$ is extremal then $%
C_{\infty }(Q,%
\mathbb{R}
)$ is universally complete $K$-space whose base is isomorphic to the algebra
of all clopen subsets in $Q$. The vector lattice $C(Q,%
\mathbb{R}
)$ of all continuous functions on $Q$ is an order-dense ideal in $C_{\infty
}(Q,%
\mathbb{R}
)$; thus, $C(Q,%
\mathbb{R}
)$ is a $K$-space ($K_{\sigma }$-space) if and only if such is $C_{\infty
}(Q,%
\mathbb{R}
)$. We denote 
\begin{equation*}
C_{\infty }(Q,%
\mathbb{C}
)=C_{\infty }(Q,%
\mathbb{R}
)\dotplus iC_{\infty }(Q,%
\mathbb{R}
).
\end{equation*}

Consider now a vector space $X$ and a real vector lattice $E$. We will
assume all vector lattices to be Archimedean without further stipulations. A
mapping 
\begin{equation*}
p:X\rightarrow E_{+}
\end{equation*}%
is called an ($E$-valued) vector norm if it satisfies the following axioms:

$1).$%
\begin{equation*}
p(x)=\theta \leftrightarrow x=\mathbf{0,}
\end{equation*}%
$x\in X$;

$2).$%
\begin{equation*}
p(\lambda x)=\left\vert \lambda \right\vert p(x),
\end{equation*}%
$\lambda \in 
\mathbb{R}
,$ $x\in X$;

$3).$%
\begin{equation*}
p(x+y)\leq p(x)+p(y),
\end{equation*}%
$x,y\in X.$

A vector norm $p$ is said to be a \textbf{decomposable} or \textbf{%
Kantorovich norm} if:

$4).$\ for arbitrary $e_{1},e_{2}\in E_{+}$ and $x\in X$, the equality 
\begin{equation*}
p(x)=e_{1}+e_{2},
\end{equation*}%
implies the existence of $x_{1},x_{2}\in X$ such that 
\begin{equation*}
x=x_{1}+x_{2},
\end{equation*}%
and 
\begin{equation*}
p(x_{k})=e_{k},
\end{equation*}%
$k=1,2.$

The triple $(X,p,E)$ (simpler, $X$ or $(X,p)$ with the implied parameters
omitted) is called a\textbf{\ lattice-normed space} if $p$ is an $E$-valued
norm on the vector space $X$. If the norm $p$ is decomposable then the space 
$(X,p)$ itself is called \textbf{decomposable}.

Take a net $(x_{\alpha })_{\alpha \in A}$ in $X$. We say that $(x_{\alpha })$
$bo$-converges to an element $x\in X$ and write 
\begin{equation*}
x=bo-\lim x_{\alpha },
\end{equation*}%
provided that there exists a decreasing net $(e_{\gamma })_{\gamma \in
\Gamma }$ in $E$ such that 
\begin{equation*}
\underset{\gamma \in \Gamma }{\inf }e_{\gamma }=\theta ,
\end{equation*}%
and, for every $\gamma \in \Gamma $, there exists an index $\alpha (\gamma
)\in A$ such that 
\begin{equation*}
p(x-x_{\alpha })\leq e_{\gamma },
\end{equation*}%
for all 
\begin{equation*}
\alpha \geq \alpha (\gamma ).
\end{equation*}

Let $e\in E_{+}$ be an element satisfying the following condition:

$(\ast )$ for an arbitrary $\varepsilon >0$, there exists an index $\alpha
(\varepsilon )\in A$ such that 
\begin{equation*}
p(x-x_{\alpha })\leq \varepsilon e,
\end{equation*}%
for all 
\begin{equation*}
\alpha \geq \alpha (\varepsilon ).
\end{equation*}%
Then we say that $(x_{\alpha })$ $br$\textbf{-converges} to $x$ and write 
\begin{equation*}
x=br-\lim x_{\alpha }.
\end{equation*}%
We say that a net $(x_{\alpha })$ is $bo$\textbf{-fundamental} ($br$\textbf{%
-fundamental}) if the net 
\begin{equation*}
(x_{\alpha }-x_{\beta })_{(\alpha ,\beta )\in A\times A}
\end{equation*}%
$bo$-converges ($br$-converges) to zero. A lattice-normed space is $bo$%
\textbf{-complete} ($br$\textbf{-complete}) if every $bo$-fundamental ($br$%
-fundamental) net in it $bo$-converges ($br$-converges) to some element of
the space.

Take a net $(x_{\xi })_{\xi \in \Xi }$ and relate to it a net $(y_{\alpha
})_{\alpha \in A},$ where 
\begin{equation*}
A=\mathcal{P}_{fin}(\Xi ),
\end{equation*}%
is the collection of all finite subsets of $\Xi $ and 
\begin{equation*}
y_{\alpha }=\dsum\limits_{\xi \in \alpha }x_{\xi }.
\end{equation*}%
If 
\begin{equation*}
x=bo-\lim y_{\alpha },
\end{equation*}%
exists, then we call $(x_{\xi })$ $bo$\textbf{-summable} with sum $x$ and
write 
\begin{equation*}
x=bo-\dsum\limits_{\xi \in \Xi }x_{\xi }.
\end{equation*}

A set 
\begin{equation*}
M\subset X,
\end{equation*}%
is called \textbf{bounded in norm} or \textbf{norm-bounded} if there exists $%
e\in E_{+}$ such that 
\begin{equation*}
p(x)\leq e,
\end{equation*}%
for all $x\in M$. A space $X$ is said to be $d$\textbf{-complete} if every
bounded set of pairwise disjoint elements in $X$ is $bo$-summable.

We call elements $x,y\in X$ \textbf{disjoint} and write 
\begin{equation*}
x\bot y,
\end{equation*}%
whenever 
\begin{equation*}
p(x)\wedge p(y)=\theta .
\end{equation*}%
Obviously, the relation $\bot $ satisfies all axioms of disjointness.

We call a norm $p$ (or the whole space $X$) $d$\textbf{-decomposable}
provided that, for every $x\in X$ and disjoint $e_{1},e_{2}\in E_{+},$ there
exist $x_{1},x_{2}\in X$ such that 
\begin{equation*}
x=x_{1}+x_{2},
\end{equation*}%
and 
\begin{equation*}
p(x_{k})=e_{k},
\end{equation*}%
$k=1,2.$ A decomposable $bo$-complete lattice-normed space is called a 
\textbf{Banach-Kantorovich space}.

Let $Q$ be a topological space. A \textbf{fiber bundle} over $Q$ is
understood as an arbitrary continuous surjection 
\begin{equation*}
\sigma :\mathfrak{X}\rightarrow Q,
\end{equation*}%
from a topological space $\mathfrak{X}$ onto $Q$. A non-empty set 
\begin{equation*}
\mathfrak{X}_{q}=\sigma ^{-1}(q),
\end{equation*}%
is called a \textbf{fiber} at the point $q\in Q.$ A mapping $s$ from a
non-empty set $dom(s)\subset Q$ into $\mathfrak{X}$ is called a \textbf{%
section} over $dom(s)$, if 
\begin{equation*}
s(q)\in X_{q},
\end{equation*}%
for all $q\in dom(s).$ A continuous section $s$ is called \textbf{local}, 
\textbf{almost global}, or \textbf{global}, if its domain $dom(s)$ is
respectively proper open subset, comeager subset, or coinsides with $Q$. A
fiber bundle 
\begin{equation*}
\sigma :\mathfrak{X}\rightarrow Q,
\end{equation*}%
is usually identified with the mapping 
\begin{equation*}
q\mapsto \mathfrak{X}_{q},
\end{equation*}%
$q\in Q,$ and instead of $\mathfrak{X}_{q}$ we may write $\mathfrak{X}(q)$.
Moreover, often times one denotes a fiber bundle just $\mathfrak{X}$, thus
omitting assumed parameters $\sigma $ and $Q$.

A set of sections $S$ is called \textbf{fiberwise dense} in $\mathfrak{X}$,
if the set 
\begin{equation*}
\{s(q):s\in S\},
\end{equation*}%
is dense in the space $\mathfrak{X}(q)$ for each $q\in Q.$ By a \textbf{%
product of two fiber bundles} 
\begin{equation*}
\sigma :\mathfrak{X}\rightarrow Q,
\end{equation*}%
and 
\begin{equation*}
\sigma ^{\prime }:\mathfrak{X}^{\prime }\rightarrow Q^{\prime },
\end{equation*}%
we understand a fiber bundle 
\begin{equation*}
\theta :\mathfrak{X}\times _{Q}\mathfrak{X}^{\prime }\rightarrow Q,
\end{equation*}%
such that 
\begin{equation*}
\mathfrak{X}\times _{Q}\mathfrak{X}^{\prime }=\{(x,x^{\prime })\in \mathfrak{%
X}\times _{Q}\mathfrak{X}^{\prime }:\sigma (x)=\sigma ^{\prime }(x^{\prime
})\},
\end{equation*}%
and 
\begin{equation*}
\theta :(x,x^{\prime })\mapsto \sigma (x)=\sigma ^{\prime }(x^{\prime }),
\end{equation*}%
where 
\begin{equation*}
(x,x\prime )\in \mathfrak{X}\times _{Q}\mathfrak{X}^{\prime }.
\end{equation*}

Let us now consider a fiber bundle 
\begin{equation*}
\sigma :\mathfrak{X}\rightarrow Q,
\end{equation*}%
and let 
\begin{equation*}
\mathfrak{X}_{q}=\sigma ^{-1}(q),
\end{equation*}%
be a Banach space for all $q\in Q,$ i.e. we consider a mapping 
\begin{equation*}
q\mapsto \mathfrak{X}_{q},
\end{equation*}%
from $Q$ to the class of Banach spaces. Let us turn the set of all global
sections $S(Q,\mathfrak{X})$ of the fiber bundle $\mathfrak{X}$ into a
complex vector space by setting 
\begin{equation*}
(\alpha u+\beta v)(q)=\alpha u(q)+\beta v(q),
\end{equation*}%
$q\in Q,$ $\alpha ,\beta \in 
\mathbb{C}
,$ and $u,v\in S(Q,\mathfrak{X}).$ For each section $s\in S(Q,\mathfrak{X})$
we introduce a pointwise norm 
\begin{equation*}
\left\Vert s\right\Vert :q\mapsto \left\Vert s(q)\right\Vert _{q},
\end{equation*}%
$q\in Q.$

A subset $\mathfrak{C}$ of global sections $S(Q,\mathfrak{X})$ of a fiber
bundle $\mathfrak{X}$ is called a \textbf{continuous structure} in $%
\mathfrak{X}$, if it satisfies the following conditions:

$1).$ $\mathfrak{C}$ is a subspace of $S(Q,\mathfrak{X});$

$2).$ for each $s\in \mathfrak{C},$ the function $\left\Vert s\right\Vert $
is continuous;

$3).$ for each $q\in Q,$ the set 
\begin{equation*}
\{s(q):s\in \mathfrak{C}\},
\end{equation*}%
is a dense subset in the fiber $\mathfrak{X}_{q}.$

The continuous structure $\mathfrak{C}$ allows one to define a topology in $%
\mathfrak{X}$ with the base consisting of the sets 
\begin{equation*}
\{U(s,\varepsilon )\},
\end{equation*}%
where $\varepsilon >0,$ and $s$ is a restriction of a section from $%
\mathfrak{C}$ to some open subset. With such a topology, the bundle 
\begin{equation*}
\sigma :\mathfrak{X}\rightarrow Q,
\end{equation*}%
is called a \textbf{continuous Banach fiber bundle}. Moreover, a section $%
u\in S(Q,\mathfrak{X})$ is continuous iff the function 
\begin{equation*}
q\mapsto \left\Vert u-s\right\Vert _{q},
\end{equation*}%
$q\in Q,$ is continuous for each section $s\in \mathfrak{C}.$ We denote by $%
C(Q,\mathfrak{X})$ the set of all continuous global sections of the
continuous Banach fiber bundle $\mathfrak{X}$.

Let now $Q$ be a Stonean compact, and let $\mathfrak{X}$ be a continuous
Banach fiber bundle over $Q$. If $u$ is an almost global section of the
bundle $\mathfrak{X}$, then the function 
\begin{equation*}
q\mapsto \left\Vert u(q)\right\Vert _{q},
\end{equation*}%
$q\in Q,$ is defined and continuous on a comeager subset $dom(u)$ of $Q$.
Thus, there exists a unique function $p_{u}\in C(Q,%
\mathbb{C}
),$ such that 
\begin{equation*}
p_{u}=\left\Vert u(q)\right\Vert _{q},
\end{equation*}%
$q\in dom(u)\subset Q.$ Let us introduce an equivalence relation in the set $%
M(Q,\mathfrak{X})$ of all almost global sections of the bundle $\mathfrak{X}$%
, by setting 
\begin{equation*}
u\sim v
\end{equation*}%
iff 
\begin{equation*}
u(q)=v(q),
\end{equation*}%
for all $q\in dom(u)\cap dom(v).$ For 
\begin{equation*}
u\sim v,
\end{equation*}%
we have 
\begin{equation*}
p_{u}=p_{v},
\end{equation*}%
thus, we can set by definition 
\begin{equation*}
p_{\widehat{u}}=p(\widehat{u})=p_{u},
\end{equation*}%
where $\widehat{u}$ is a class of equivalency of the almost global section $u
$. Let us denote by $C_{\infty }(Q,\mathfrak{X})$ the factor-set 
\begin{equation*}
M(Q,\mathfrak{X})\diagup \sim .
\end{equation*}

The set $C_{\infty }(Q,\mathfrak{X})$ can naturally be endowed with the
structure of a lattice-normed space. Indeed (see for example \cite%
{Kusraev2000}), by a sum 
\begin{equation*}
\widehat{u}+\widehat{v},
\end{equation*}%
of $u,v\in C_{\infty }(Q,\mathfrak{X}),$ we understand a class of
equivalence of the almost global section 
\begin{equation*}
q\mapsto u(q)+v(q),
\end{equation*}%
$q\in dom(u)\cap dom(v),$ etc.

In each class of equivalence $\widehat{u}$ there exists a unique section 
\begin{equation*}
\widetilde{u}\in \widehat{u},
\end{equation*}%
such that 
\begin{equation*}
dom(v)\subset dom(\widetilde{u}),
\end{equation*}%
for all 
\begin{equation*}
v\in \widehat{u}.
\end{equation*}%
Such a section $\widetilde{u}$ is called \textbf{extended}. Therefore, we
can think of the space $C_{\infty }(Q,\mathfrak{X})$ as a space of all
extended almost global sections of the fiber bundle $\mathfrak{X}$.

A continuous fiber bundle $\mathfrak{X}$ over Stonean compact $Q$ is called 
\textbf{complete}, if each its bounded almost global continuous section can
be extended to a global continuous section. We denote 
\begin{equation*}
C_{\#}(Q,\mathfrak{X})=\{u\in C_{\infty }(Q,\mathfrak{X}):p_{u}\in C(Q,%
\mathbb{R}
)\}.
\end{equation*}%
The following conditions are equivalent (see for example \cite{Gutman1993}):

$1).$ the fiber bundle $\mathfrak{X}$ is complete;

$2).$ 
\begin{equation*}
C_{\#}(Q,\mathfrak{X})=C(Q,\mathfrak{X});
\end{equation*}

$3).$ $C(Q,\mathfrak{X})$ is a Banach-Kantorovich space;

$4).$ each bounded continuous section of the fiber bundle $\mathfrak{X}$,
which is defined on a dense subset of $Q$, can be extended to a global
continuous section.

Two complete continuous Banach fiber bundles $\mathfrak{X}$ and $\mathfrak{Y}
$ over the same $Q$, where $Q$ is a Stonean compact, are $Q$-isomorphic iff $%
C_{\infty }(Q,\mathfrak{X})$ and $C_{\infty }(Q,\mathfrak{Y})$ are $Q$%
-isomorphic (see \cite{Gutman1993} for details).

\section{Continuous fiber bundles of $C^{\ast }$-algebras over Stonean
compact $Q$}

Let us give a formal definition of $C^{\ast }$-algebra over $C_{\infty }(Q,%
\mathbb{C}
),$ where $Q$ be a Stonean compact.

\begin{definition}
Let $\mathfrak{A}$ be a $^{\ast }$-algebra over $%
\mathbb{C}
.$ Let in addition $\mathfrak{A}$ be a module over $C_{\infty }(Q,%
\mathbb{C}
),$ such that 
\begin{equation*}
(fu)^{\ast }=\overline{f}u^{\ast },
\end{equation*}%
\begin{equation*}
(fu)v=f(uv)=u(fv),
\end{equation*}%
for all $f\in C_{\infty }(Q,%
\mathbb{C}
),$ $u,v\in \mathfrak{A}.$ Let $\mathfrak{A}$ be endowed with a $C_{\infty
}(Q,%
\mathbb{C}
)$-valued norm $\left\Vert .\right\Vert ,$ which turns $\mathfrak{A}$ into a
Banach-Kantorovich space, such that 
\begin{equation*}
\left\Vert fu\right\Vert =\left\vert f\right\vert \left\Vert u\right\Vert ,
\end{equation*}%
for all $f\in C_{\infty }(Q,%
\mathbb{C}
),u\in \mathfrak{A}.$ $\mathfrak{A}$ is called a $C^{\ast }$\textbf{-algebra
over }$C_{\infty }(Q,%
\mathbb{C}
),$ where $Q$ be a Stonean compact, if the following conditions are
satisfied for all $u,v\in \mathfrak{A}:$

$1).$ 
\begin{equation*}
\left\Vert u\cdot v\right\Vert \leq \left\Vert u\right\Vert \left\Vert
v\right\Vert ;
\end{equation*}

$2).$ 
\begin{equation*}
\left\Vert u^{\ast }\right\Vert =\left\Vert u\right\Vert ;
\end{equation*}

$3).$ 
\begin{equation*}
\left\Vert u^{\ast }\cdot u\right\Vert =\left\Vert u\right\Vert ^{2}.
\end{equation*}
\end{definition}

\begin{example}
Let $\mathfrak{A}$ be an algebra of all bounded $C_{\infty }(Q,%
\mathbb{C}
)$-linear operators defined on $C_{\infty }(Q,%
\mathbb{C}
)$-Hilbert spaces. Then $\mathfrak{A}$ is an example of a $C^{\ast }$%
-algebra over\textbf{\ }$C_{\infty }(Q,%
\mathbb{C}
),$ where $Q$ be a Stonean compact.
\end{example}

\begin{example}
Let $\mathfrak{A}$ be an algebra of all bounded $C_{\infty }(Q,%
\mathbb{C}
)$-linear operators defined on $C_{\infty }(Q,%
\mathbb{C}
)$-Hilbert spaces, and $\mathfrak{B}$ be its $^{\ast }$-subalgebra closed in 
$C_{\infty }(Q,%
\mathbb{C}
)$-valued norm. Then $\mathfrak{B}$ is an example of a $C^{\ast }$-algebra
over\textbf{\ }$C_{\infty }(Q,%
\mathbb{C}
),$ where $Q$ be a Stonean compact.
\end{example}

Let $\mathfrak{X}$ be a mapping 
\begin{equation*}
\mathfrak{X}:q\in Q\longmapsto \mathfrak{X}(q),
\end{equation*}

where $\mathfrak{X}(q)$ be a $C^{\ast }$-algebra for all $q\in Q,\mathfrak{\ 
}$where $Q$ be a Stonean compact.

Let $\mathfrak{C}$ be a certain collection of almost global sections of $%
\mathfrak{X}$.

\begin{definition}
A pair $(\mathfrak{X},\mathfrak{C})$ is called \textbf{a} \textbf{continuous
fiber bundle of }$C^{\ast }$\textbf{-algebras over Stonean compact }$Q$, if
the following conditions are satisfied:

$1).$ $\mathfrak{X}(q)$ is a $C^{\ast }$-algebra for all $q\in Q;$

$2).$ the pair $(\mathfrak{X},\mathfrak{C})$ is continuous Banach fiber
bundle over $Q;$

$3).$ if 
\begin{equation*}
u,v\in \mathfrak{C},
\end{equation*}%
then 
\begin{equation*}
u\cdot v\in \mathfrak{C},
\end{equation*}%
where 
\begin{equation*}
u\cdot v:q\in dom(u)\cap dom(v)\mapsto u(q)\cdot v(q);
\end{equation*}

$4).$ if 
\begin{equation*}
u\in \mathfrak{C},
\end{equation*}%
then 
\begin{equation*}
u^{\ast }\in \mathfrak{C},
\end{equation*}%
where 
\begin{equation*}
u^{\ast }:q\in dom(u)\mapsto u(q)^{\ast }.
\end{equation*}
\end{definition}

The following Theorem is the first main result of the present paper.

\begin{theorem}
Let $\mathfrak{X}$ be a continuous fiber bundle of $C^{\ast }$-algebras over 
$Q$, where $Q$ be a Stonean compact. Then $C_{\infty }(Q,\mathfrak{X})$ is a 
$C^{\ast }$-algebra over $C_{\infty }(Q,%
\mathbb{C}
).$
\end{theorem}

\begin{proof}
From \cite{Gutman1993} it follows that $C_{\infty }(Q,\mathfrak{X})$ is a
Banach-Kantorovich space.

Because $\mathfrak{X}(q)$ is a $^{\ast }$-algebra for all $q\in Q,$ one can
see that $C_{\infty }(Q,\mathfrak{X})$ is a $^{\ast }$-algebra.

From the fact that for all $q\in Q,$ $\mathfrak{X}(q)$ is a $C^{\ast }$%
-algebra it follows that 
\begin{equation*}
\left\Vert \widehat{u}\cdot \widehat{v}\right\Vert =\left\Vert \widehat{%
u(q)\cdot v(q)}\right\Vert _{q}\leq \left\Vert \widehat{u(q)}\right\Vert
_{q}\cdot \left\Vert \widehat{v(q)}\right\Vert _{q}=\left\Vert \widehat{u}%
\right\Vert \cdot \left\Vert \widehat{v}\right\Vert ,
\end{equation*}%
\begin{equation*}
\left\Vert \widehat{u}^{\ast }\right\Vert =\left\Vert \widehat{u(q)^{\ast }}%
\right\Vert _{q}=\left\Vert \widehat{u(q)}\right\Vert _{q}=\left\Vert 
\widehat{u}\right\Vert ,
\end{equation*}%
and 
\begin{eqnarray*}
\left\Vert \widehat{u}^{\ast }\cdot \widehat{u}\right\Vert &=&\left\Vert 
\widehat{u(q)^{\ast }\cdot u(q)}\right\Vert _{q}=\left\Vert \widehat{%
u(q)^{\ast }}\cdot \widehat{u(q)}\right\Vert _{q}= \\
&=&\left\Vert \widehat{u(q)}^{2}\right\Vert _{q}=\left\Vert \widehat{u(q)}%
\right\Vert _{q}^{2}=\left\Vert \widehat{u}\right\Vert ^{2}.
\end{eqnarray*}%
Thus, $C_{\infty }(Q,\mathfrak{X})$ is a $C^{\ast }$-algebra over $C_{\infty
}(Q,%
\mathbb{C}
),$ where $Q$ be a Stonean compact.
\end{proof}

\section{Representations of $C^{\ast }$-algebras over\textbf{\ }$C_{\infty
}(Q,%
\mathbb{C}
)$, where $Q$ is a Stonean compact, as Complete Continuous Fiber Bundles of $%
C^{\ast }$-algebras over $Q$}

Let now $\mathfrak{X}$ and $\mathfrak{Y}$ be continuous fiber bundles of $%
C^{\ast }$-algebras over the same Stonean compact $Q$. A mapping 
\begin{equation*}
H:q\in Q\mapsto H(q),
\end{equation*}%
$q\in Q,$ where 
\begin{equation*}
H(q):\mathfrak{X}(q)\rightarrow \mathfrak{Y}(q),
\end{equation*}%
is an injective $^{\ast }$-homomorphism of $C^{\ast }$-algebras, is called \
a $Q$\textbf{-}$C^{\ast }$\textbf{-embedding} of $\mathfrak{X}$ in $%
\mathfrak{Y}$, if 
\begin{equation*}
\{Hu:u\in M(Q,\mathfrak{X})\}\subset M(Q,\mathfrak{Y}).
\end{equation*}%
In the case when 
\begin{equation*}
\{Hu:u\in M(Q,\mathfrak{X})\}=M(Q,\mathfrak{Y}),
\end{equation*}%
the $Q$-$C^{\ast }$-imbedding $H$ is called a $Q$\textbf{-}$C^{\ast }$%
\textbf{-isomorphism,} and in this case the fiber bundles of $C^{\ast }$%
-algebras $\mathfrak{X}$ and $\mathfrak{Y}$ are called $Q$\textbf{-}$C^{\ast
}$\textbf{-isomorphic}.

The following Theorem is the second main result of the current paper.

\begin{theorem}
Let $\mathfrak{A}$ be a $C^{\ast }$-algebras over $C_{\infty }(Q,%
\mathbb{C}
),$ where $Q$ is a Stonean compact. Then there exists a unique (up to a $Q$-$%
C^{\ast }$-isomorphism) complete continuous fiber bundle $\mathfrak{X}$ of $%
C^{\ast }$-algebras over $Q$ such that $\mathfrak{A}$ is $Q$-$C^{\ast }$%
-isomorphic to $C_{\infty }(Q,\mathfrak{X}).$
\end{theorem}

\begin{proof}
Let us set 
\begin{equation*}
\mathfrak{D}=\{u\in \mathfrak{A}:\left\Vert u\right\Vert \in C(Q,%
\mathbb{R}
)\}.
\end{equation*}%
One can see that $\mathfrak{D}$ is a $C_{\infty }(Q,\mathfrak{X})$-module,
which is $bo$-dense in $\mathfrak{A}$. In addition, it is easy to see that $%
\mathfrak{B}$ is a $^{\ast }$-algebra, and 
\begin{equation*}
\left\Vert u^{\ast }u\right\Vert =\left\Vert u\right\Vert ^{2},
\end{equation*}%
for all $u\in \mathfrak{D}.$

Let us define a $C_{\infty }(Q,%
\mathbb{R}
)$-valued seminorm $\rho _{q}$ on $\mathfrak{D}$ as 
\begin{equation*}
\rho _{q}(u)=p_{\left\Vert u\right\Vert }(q),
\end{equation*}%
for all $q\in Q.$

Let 
\begin{equation*}
J_{q}^{0}=\{u\in \mathfrak{D}:\rho _{q}(u)=0\}.
\end{equation*}%
Consider the factor-algebras 
\begin{equation*}
\mathfrak{D}_{q}=\mathfrak{D}\diagup J_{q}^{0},
\end{equation*}%
and let $\left\Vert .\right\Vert _{q}$ be the norm on $\mathfrak{D}_{q}$
generated by the seminorm $\rho _{q}.$

Let 
\begin{equation*}
\pi _{q}:\mathfrak{D}\rightarrow \mathfrak{D}_{q},
\end{equation*}%
$q\in Q,$ be a projection from $\mathfrak{D}$ to $\mathfrak{D}_{q}.$ Then 
\begin{equation*}
\pi _{q}(u\cdot v)=\pi _{q}(u)\cdot \pi _{q}(v),
\end{equation*}%
and 
\begin{equation*}
\pi _{q}(u^{\ast })=\pi _{q}(u)^{\ast }.
\end{equation*}

Because 
\begin{equation*}
\left\Vert \pi _{q}(u)\right\Vert _{q}=\rho _{q}(u),
\end{equation*}%
for all $u\in \mathfrak{D},$ and $q\in Q,$ it follows that 
\begin{equation*}
\left\Vert \pi _{q}(u)\cdot \pi _{q}(v)\right\Vert _{q}=\left\Vert \pi
_{q}(u\cdot v)\right\Vert _{q}=\rho _{q}(u\cdot v)=
\end{equation*}%
\begin{equation*}
=p_{\left\Vert u\cdot v\right\Vert }(q)\leq p_{\left\Vert u\right\Vert \cdot
\left\Vert v\right\Vert }(q)=p_{\left\Vert u\right\Vert }(q)\cdot
p_{\left\Vert v\right\Vert }(q)=
\end{equation*}%
\begin{equation*}
=\rho _{q}(u)\cdot \rho _{q}(v)=\left\Vert \pi _{q}(u)\right\Vert _{q}\cdot
\left\Vert \pi _{q}(v)\right\Vert _{q}.
\end{equation*}

Similarly, 
\begin{equation*}
\left\Vert \pi _{q}(u)^{\ast }\right\Vert _{q}=\left\Vert \pi _{q}(u^{\ast
})\right\Vert _{q}=\rho _{q}(u^{\ast })=
\end{equation*}%
\begin{equation*}
=p_{\left\Vert u^{\ast }\right\Vert }(q)=p_{\left\Vert u\right\Vert
}(q)=\rho _{q}(u)=\left\Vert \pi _{q}(u)\right\Vert _{q},
\end{equation*}%
for all $u\in \mathfrak{D},$ and $q\in Q.$

In the similar manner we get that 
\begin{equation*}
\left\Vert \pi _{q}(u)\cdot \pi _{q}(u)^{\ast }\right\Vert _{q}=\left\Vert
\pi _{q}(u\cdot u^{\ast })\right\Vert _{q}=\rho _{q}(u\cdot u^{\ast
})=p_{\left\Vert u\cdot u^{\ast }\right\Vert }(q)=
\end{equation*}%
\begin{equation*}
=p_{\left\Vert u\right\Vert ^{2}}(q)=p_{\left\Vert u\right\Vert
}(q)^{2}=\rho _{q}(u)^{2}=\left\Vert \pi _{q}(u)\right\Vert _{q}^{2},
\end{equation*}%
for all $u\in \mathfrak{D},$ and $q\in Q.$

So, one can see that $(\mathfrak{D}_{q},\left\Vert .\right\Vert _{q})$ is a
pre-$C^{\ast }$-algebra. Let us denote a completion of $\mathfrak{D}_{q}$ in
the norm $\left\Vert .\right\Vert _{q}$ by $\mathfrak{X}_{q}$. From \cite%
{Dixmier1977} it follows that $\mathfrak{X}_{q}$ is a $C^{\ast }$-algebra
for each $q\in Q.$

Let 
\begin{equation*}
i_{q}:\mathfrak{D}_{q}\rightarrow \mathfrak{X}_{q},
\end{equation*}%
be the canonical embedding of the pre-$C^{\ast }$-algebra $\mathfrak{D}_{q}$
into the $C^{\ast }$-algebra $\mathfrak{X}_{q},$ for each $q\in Q.$ From 
\cite{Dixmier1977} it follows that 
\begin{equation*}
i_{q}(x\cdot y)=i_{q}(x)\cdot i_{q}(y),
\end{equation*}%
and 
\begin{equation*}
i_{q}(x^{\ast })=i_{q}(x)^{\ast },
\end{equation*}%
for each $x,y\in \mathfrak{D}_{q},$ and $q\in Q.$ Thus, the mapping 
\begin{equation*}
\varphi _{q}=\pi _{q}\circ i_{q},
\end{equation*}%
for each $q\in Q$ is a $^{\ast }$-homomorphism from $\mathfrak{D}_{q}$ into $%
\mathfrak{X}_{q}.$

Let us define a mapping 
\begin{equation*}
\mathfrak{X}:q\longmapsto \mathfrak{X}_{q},
\end{equation*}%
for each $q\in Q.$

Let $\mathfrak{C}$ be the set of all sections $\widetilde{u}$, such that 
\begin{equation*}
\widetilde{u}(q)=\varphi _{q}(u),
\end{equation*}%
$u\in \mathfrak{D},$ $q\in Q.$ It is easy to see that $(\mathfrak{X},%
\mathfrak{C})$ is a continuous Banach fiber bundle, and from definition of
the set $\mathfrak{C}$ it follows that $(\mathfrak{X},\mathfrak{C})$ is a
continuous fiber bundle of $C^{\ast }$-algebras over $Q,$ where $Q$ is a
Stonean compact.

Let us now consider $C_{\infty }(Q,\mathfrak{X}).$ From Theorem 1 it follows
that $C_{\infty }(Q,\mathfrak{X})$ is a $C^{\ast }$-algebra over $C_{\infty
}(Q,%
\mathbb{C}
),$ with a $C_{\infty }(Q,%
\mathbb{R}
)$-valued norm $\left\Vert .\right\Vert _{C_{\infty }(Q,\mathfrak{X})}.$

Now we show that $\mathfrak{A}$ is $Q$-$C^{\ast }$-isomorphic to $C_{\infty
}(Q,\mathfrak{X}).$ Indeed, for each $u\in \mathfrak{D},$ we set 
\begin{equation*}
\Phi _{0}(u)=\widehat{\widetilde{u}}.
\end{equation*}%
One can easily see that $\Phi _{0}$ is an $Q$-isometry. In addition, $\Phi
_{0}$ satisfies the following identities: 
\begin{equation*}
\Phi _{0}(u\cdot v)=\widehat{\widetilde{u\cdot v}}=\widehat{\varphi
_{q}(u\cdot v)}=\widehat{\varphi _{q}(u)\cdot \varphi _{q}(v)}=
\end{equation*}%
\begin{equation*}
=\widehat{\varphi _{q}(u)}\cdot \widehat{\varphi _{q}(v)}=\widehat{%
\widetilde{u}}\cdot \widehat{\widetilde{v}}=\Phi _{0}(u)\cdot \Phi _{0}(v),
\end{equation*}%
and 
\begin{equation*}
\Phi _{0}(u^{\ast })=\widehat{\widetilde{u^{\ast }}}=\widehat{\varphi
_{q}(u^{\ast })}=
\end{equation*}%
\begin{equation*}
=\widehat{\varphi _{q}(u)}^{\ast }=(\widehat{\widetilde{u}})^{\ast }=\Phi
_{0}(u)^{\ast },
\end{equation*}%
for all $u,v\in \mathfrak{D,}$ and $q\in Q.$

From \cite{Gutman1993} it follows that $\Phi _{0}$ can be extended to an $%
C_{\infty }(Q,%
\mathbb{C}
)$-module $Q$-isomeiric isomorphism 
\begin{equation*}
\Phi :\mathfrak{A}\rightarrow C_{\infty }(Q,\mathfrak{X}).
\end{equation*}%
One can as well see that $\Phi $ preserves the operations of multiplication
and involution, i.e. $\Phi $ is an $Q$-$C^{\ast }$-isomorphism from $%
\mathfrak{A}$ onto $C_{\infty }(Q,\mathfrak{X}).$

From \cite{Gutman1993} it follows that $\mathfrak{X}$ is a complete
continuous fiber bundle of $C^{\ast }$-algebras over $Q$, where $Q$ is a
Stonean compact.

To establish the uniqueness one must recall (see Preliminaries) that two
complete continuous Banach fiber bundles $\mathfrak{X}$ and $\mathfrak{Y}$
over the same $Q$, where $Q$ is a Stonean compact, are $Q$-isomorphic iff $%
C_{\infty }(Q,\mathfrak{X})$ and $C_{\infty }(Q,\mathfrak{Y})$ are $Q$%
-isomorphic.
\end{proof}

\end{document}